\documentclass{amsart}
\usepackage{graphicx}
\usepackage{amsfonts}
\usepackage{float}
\newtheorem{tm}{Theorem}

\newtheorem{defi}{Definition}

\newtheorem{rem}{Remark}
\newtheorem{rems}{Remarks}
\newtheorem{lm}{Lemma}

\newtheorem{prop}{Proposition}
\newtheorem{nota}{Notation}
\newtheorem{quest}{Question}

\begin{document}
\title{Degree $6$ hyperbolic polynomials and orders of moduli}
\author{Yousra Gati, Vladimir Petrov Kostov and Mohamed Chaouki Tarchi}

\address{Universit\'e de Carthage, EPT-LIM, Tunisie}
\email{yousra.gati@gmail.com}
\address{Universit\'e C\^ote d’Azur, CNRS, LJAD, France}
\email{vladimir.kostov@unice.fr}
\address{Universit\'e de Carthage, EPT-LIM, Tunisie}
\email{mohamedchaouki.tarchi@gmail.com}
\begin{abstract}
  We consider real univariate degree $d$
  real-rooted polynomials with non-vanishing coefficients. 
  Descartes' rule of signs implies that such a polynomial
  has $\tilde{c}$ positive and
  $\tilde{p}$ negative roots counted with multiplicity, where $\tilde{c}$ and
  $\tilde{p}$ are the numbers of sign changes and sign preservations in the
  sequence of its coefficients, $\tilde{c}+\tilde{p}=d$. For $d=6$,
  we give the exhaustive answer to the question: When the moduli of all $6$
  roots are distinct and arranged on the real positive half-axis,
  in which positions can the moduli of the negative roots be depending on
  the signs of the coefficients?


\end{abstract}
\maketitle

\section{Introduction}

A real univariate polynomial is {\em hyperbolic} if all its roots are real. We
consider hyperbolic polynomials with all coefficients non-vanishing. For such
a degree $d$ polynomial, the
classical Descartes' rule of signs implies that the number of its positive
(resp. negative) roots counted with multiplicity is equal to the number
$\tilde{c}$ of sign changes (resp. $\tilde{p}$ of sign
preservations) in the sequence of its coefficients, see
\cite{Ca}, \cite{Cu}, \cite{DG},
\cite{Des}, \cite{Fo}, \cite{Ga}, \cite{J}, \cite{La} or \cite{Mes};
$\tilde{c}+\tilde{p}=d$. This fact,
however, does not answer the following more subtle question:

\begin{quest}\label{quest1}
  For fixed degree $d$, consider the set of hyperbolic polynomials
  with given signs of the
  coefficients and with distinct moduli of roots. Suppose that
  these moduli are arranged in a string
  on the real positive half-axis. Then in
  which positions can the moduli of the negative roots be
  depending on the signs of the coefficients?
\end{quest}

We give the exhaustive answer to the question for $d=6$. For $d\leq 5$, its
answer can be found in \cite{KoCMA23}, see Example~1.1 and Section~3 therein.
In order to recall some other results directly related to
Question~\ref{quest1} we remind the following definition:

\begin{defi}
  {\rm (1) A real polynomial $Q:=\sum _{j=0}^{d}q_jx^j$
  is said to define the {\em sign pattern}
  $\sigma (Q):=({\rm sgn}(q_d)$, $\ldots$,
  ${\rm sgn}(q_0))$. Formally, a sign pattern of length $d+1$
  is a string of $d+1$ signs $+$
  and/or $-$. We operate mainly with sign patterns beginning with a~$+$.
  Thus a sign pattern is completely defined by the corresponding
  {\em change-preservation pattern} (and vice versa)
  which is a $d$-vector whose
  components are the letters $p$ and $c$; when
  $q_jq_{j-1}>0$ (resp. $q_jq_{j-1}<0$), in the $j$th position from the right
  there is a $p$ (resp. a $c$).
  \vspace{1mm}
  
  (2) The {\em order of moduli} defined by the roots of a given hyperbolic
  polynomial $Q$ is denoted as follows. (The general definition should be
  clear from this example.) Suppose that $d=6$ and that there are
  three negative roots $-\gamma _3<-\gamma _2<-\gamma _1$
  and three positive roots $\alpha _1<\alpha _2<\alpha _3$
  (so $\tilde{c}=\tilde{p}=3$), where

  $$\alpha _1<\gamma _1<\gamma _2<\alpha _2<\gamma _3<\alpha _3~.$$
  Then we say that the roots define the order of moduli $PNNPNP$, i.~e. 
  the letters $P$ and $N$ denote the relative positions of the moduli of
  positive and negative roots.
  \vspace{1mm}
  
  (3) For a given degree $d$, a couple (change-preservation pattern,
  order of moduli) (further we say {\em couple} for short) is {\em compatible} 
  if the number of letters $c$ (resp. $p$)
  of the former is equal to the number of letters $P$ (resp. $N$)
  of the latter. A compatible couple is {\em realizable} if there exists a
  hyperbolic polynomial whose coefficients (resp. moduli of roots) define the
  change-preservation pattern (resp. the order of moduli) of the couple.} 
    \end{defi}

We can give now a more precise formulation of Question~\ref{quest1}:

\begin{quest}\label{quest2}
  For a given degree $d$, which compatible couples are realizable?
\end{quest}

There are two extremal situations with regard to Question~\ref{quest2}.

\begin{defi}
  {\rm For a given change-preservation pattern (or, equivalently,
    a sign pattern) one defines the corresponding {\em canonical order of
      moduli} as follows. One reads the pattern from the right and one
    writes the order from the left. To each letter $c$ (resp. $p$) one puts
    in correspondence the letter $P$ (resp. $N$).}
\end{defi}

Each sign pattern (or equivalently change-preservation pattern)
is realizable with its corresponding canonical order, see
\cite[Proposition~1]{KoSe}.

\begin{defi}
    {\rm A change-preservation pattern (or a sign pattern) is {\em canonical}
    if it is realizable only with the corresponding canonical order of moduli.}
\end{defi}

It is shown in \cite[Theorem~7]{KoReMa22} that a sign pattern
is canonical if and only if
it contains no four
consecutive signs $(+,+,-,-)$, $(+,-,-,+)$, $(-,-,+,+)$ or $(-,+,+,-)$. Hence
a change-preservation pattern is canonical if and only if it contains no
string $cpc$ or $pcp$. Canonical sign patterns are exceptional in the sense
that the ratio of their number and the number of all
sign patterns tends to $0$ as $d$ tends to $\infty$, see
\cite[Proposition~10]{KoReMa22}. 

The second extremal situation is the one of {\em rigid} orders of moduli. 

\begin{defi}\label{defirigid}
  {\rm An order of moduli is {\em rigid} if all hyperbolic polynomials with
    this order of moduli define one and the same sign pattern.}
\end{defi}

It is proved that (see \cite[Theorem~8]{KorigMO}) rigid are exactly
the orders of
moduli of the form $PNPNPN\cdots$, $NPNPNP\cdots$, $PP\cdots P$ or
$NN\cdots N$. The corresponding change-preservation patterns are of the form
$\cdots pcpcpc$, $\cdots cpcpcp$, $cc\cdots c$ or $pp\cdots p$. Hence rigid
orders of moduli are also exceptional.

We introduce now the $\mathbb{Z}_2\times \mathbb{Z}_2$-action:

\begin{defi}\label{defiZ2Z2}
  {\rm (1) For a given degree $d$, there are two commuting involutions
    which act on the set of couples. These are}

  $$i_m~:~Q(x)\mapsto (-1)^dQ(-x)~~~\, {\rm and}~~~\, i_r~:~Q(x)\mapsto
  x^dQ(1/x)/Q(0)~.$$
  {\rm The role of the factors $(-1)^d$ and $1/Q(0)$ is to preserve the set
    of monic polynomials. The involution $i_m$ exchanges the letters
    $P$ and $N$ in the order of moduli, the letters $c$ and $p$ in the
    change-preservation pattern and the quantities $\tilde{c}$ and
    $\tilde{p}$. The involution $i_r$ reads orders, 
    patterns and polynomials (modulo the factor $1/Q(0)$) from the right.
    It preserves the quantities $\tilde{c}$ and~$\tilde{p}$.
    \vspace{1mm}

    (2) The {\em orbits} of couples under the
    $\mathbb{Z}_2\times \mathbb{Z}_2$-action
    are of length $4$ or $2$. One can consider orbits also only of sign
    patterns or of orders of moduli.}
\end{defi}

\begin{rems}\label{remsorbit}
  {\rm (1) Given any sign pattern $\sigma$
    its orbit can
    be of length $2$ only if either $i_r(\sigma )=\sigma$ or
    $i_ri_m(\sigma )=\sigma$. Indeed, one always has $i_m(\sigma )\neq \sigma$.
    All couples of a given orbit are simultaneously (non-)realizable.
    \vspace{1mm}

(2) In the text we use the following notation -- if a sign pattern consists
of $m_1$ pluses followed by $m_2$ minuses followed by $m_3$ pluses etc., then
we denote this sign pattern by $\Sigma _{m_1,m_2,m_3,\ldots}$. For $d=6$,
an example of an orbit of a sign pattern of length $2$
is the one of $\Sigma _{3,1,3}$ with $\tilde{c}=2$, $\tilde{p}=4$ and
$i_r(\Sigma _{3,1,3})=\Sigma _{3,1,3}$. The other sign pattern of the orbit is
$i_m(\Sigma _{3,1,3})=i_mi_r(\Sigma _{3,1,3})=\Sigma_{1,1,3,1,1}$,
with $\tilde{c}=4$, $\tilde{p}=2$.}
\end{rems}

The involution $i_m$ exchanging the quantities $\tilde{c}$ and $\tilde{p}$,
when studying the realizability of the couples with $d=6$ it suffices to
consider the cases $\tilde{c}=0$, $1$, $2$ and $3$. The first three of them
have been thoroughly analysed in~\cite{KoSoz19}
(we recall the corresponding results in Section~\ref{secprelim}),
so we concentrate on the case $\tilde{c}=3$. 

\begin{lm}\label{lm1}
For $d=6$, there are $7$ orbits of sign patterns with three sign
changes:

$$\begin{array}{cclccl}
A&:&\{ \Sigma _{3,1,2,1},~\Sigma _{1,2,1,3},~\Sigma _{2,3,1,1},~\Sigma
_{1,1,3,2}\} ~,&D&:&\{\Sigma _{4,1,1,1},~\Sigma _{1,1,1,4}\}~, \\ \\
B&:&\{ \Sigma _{1,4,1,1},~
\Sigma _{1,1,4,1},~\Sigma _{3,1,1,2},~\Sigma _{2,1,1,3}\}~,&E&:&\{ \Sigma
_{2,2,2,1},~\Sigma _{1,2,2,2}\}~, \\ \\ 
C&:&\{ \Sigma _{2,1,2,2},~\Sigma _{2,2,1,2},~
\Sigma _{1,2,3,1},~\Sigma _{1,3,2,1}\}~,&F&:&\{ \Sigma _{3,2,1,1},~\Sigma
_{1,1,2,3}\} \\ \\
&&{\rm and}&G&:&\{ \Sigma _{1,3,1,2},~\Sigma _{2,1,3,1}\} ~.\end{array}$$
Out of these, canonical are exactly $B$, $D$ and $G$.
\end{lm}

\begin{rems}\label{remsBDG}
  {\rm (1) For $\sigma =\Sigma _{4,1,1,1}$, $\Sigma _{2,2,2,1}$,
    $\Sigma _{3,2,1,1}$ and $\Sigma _{1,3,1,2}$, one has $i_mi_r(\sigma )=\sigma$.


    (2) The orbits $B$, $D$ and $G$ being canonical,
    they give rise to the following
    realizable couples and only to them:}

  $$\begin{array}{llcl}
    B~:&(\Sigma _{1,4,1,1}~,~PPNNNP)~,&&(\Sigma _{1,1,4,1}~,~PNNNPP)~,\\ \\
    &(\Sigma _{3,1,1,2}~,~NPPPNN)~&{\rm and}&(\Sigma _{2,1,1,3}~,~NNPPPN)~;\\ \\
    D~:&(\Sigma _{4,1,1,1}~,~PPPNNN)&{\rm and}&(\Sigma _{1,1,1,4}~,~NNNPPP)~;\\ \\
    G~:&(\Sigma _{1,3,1,2}~,~NPPNNP)&{\rm and}&(\Sigma _{2,1,3,1}~,~PNNPPN)~.
  \end{array}$$
  \end{rems}

\begin{proof}[Proof of Lemma~\ref{lm1}]
Among the sign patterns of the form $\Sigma_{m_1,m_2,m_3,m_4}$ with
$m_1+m_2+m_3+m_4=7$ and $\tilde{c}=\tilde{p}=3$, all
components $m_i$ must be $\leq 4$. Hence there are exactly
four such sign patterns in which
exactly one component
equals $4$ (the other components equal $1$), exactly twelve in which one
component equals $3$ and exactly
four in which three components equal $2$. These are all the $20$ sign
patterns listed in the lemma. 
The last statement of the lemma is checked straightforwardly.
\end{proof}

Part (2) of Remarks~\ref{remsBDG} settling the cases $B$, $D$ and $G$, we
finish the study of realizability of couples with $d=6$, $\tilde{c}=3$ by 
Theorem~\ref{tmmain}. We remind that by Definition~\ref{defiZ2Z2} and part (1)
of Remarks~\ref{remsorbit} it suffices to give the answer only for one sign
pattern from each of the cases $A$, $C$, $E$ and $F$.

\begin{tm}\label{tmmain}
  (1) The sign pattern $\Sigma _{3,1,2,1}$ is realizable by and only by
  the orders of moduli $PPPNNN$, $PPNPNN$, $PPNNPN$, $PNPPNN$, and $NPPPNN$.
  \vspace{1mm}

  (2) The sign pattern $\Sigma _{2,1,2,2}$ is realizable by and only by
  the orders of moduli  $PNNPPN$, $NPPPNN$, $NPPNPN$, $NPPNNP$, 
  $NPNPPN$ and $NNPPPN$.
  \vspace{1mm}

  (3) The sign pattern $\Sigma _{2,2,2,1}$ is not realizable by and only by
  the following compatible orders of moduli: $NPNPNP$, $NPNNPP$,
  $NNPPNP$, $NNPNPP$, and $NNNPPP$.
  \vspace{1mm}

  (4) The sign pattern $\Sigma _{3,2,1,1}$ is realizable by and only by
  the orders of moduli $PPPNNN$, $PPNPNN$, $PPNNPN$ and $PNPPNN$.
  \vspace{1mm}
  \end{tm}

The theorem is proved in Section~\ref{secprtmmain}. The method of its proof
and some comments on the theorem are given in Section~\ref{secprelim}.

\section{Comments and the method of proof
  of Theorem~\protect\ref{tmmain}\protect\label{secprelim}}

\subsection{Systems of linear differential equations}

Hyperbolic are often the characteristic polynomials of linear systems of
ordinary differential equations. Consider such a system $dX/dt=AX$, where
$A$ is a real constant $n\times n$-matrix. Suppose that all its eigenvalues
$\lambda _1$, $\ldots$, $\lambda _n$ are
real. This is true, in particular, for symmetric matrices. Suppose also
that they are distinct. Then any component of any solution is of the form
$\sum c_je^{\lambda _jt}$, $c_j\in \mathbb{R}$. For a generic solution, all
coefficients $c_j$ are non-zero.

If the characteristic polynomial of $A$ defines a canonical sign pattern,
then one knows whether the eigenvalue of largest modulus is positive or
negative. Hence one knows (without computing the eigenvalues) whether
a generic solution grows faster in modulus as $t\rightarrow +\infty$
or as $t\rightarrow -\infty$.

\subsection{The results for $d\leq 5$}

We begin by reminding that for $d=6$,
there is just one change-preservation pattern with $\tilde{c}=0$. This is
$pppppp$ and it is realizable with the only compatible order of moduli
$NNNNNN$.

\begin{nota}
  {\rm For $d=6$ and $\tilde{c}=1$ (resp. $\tilde{c}=2$), we denote by
    $u_1$ and $u_2$ (resp. $u_1$, $u_2$ and 
    $u_3$) the number of moduli of negative roots belonging to the
    respective intervals 
    $(0,\alpha _1)$ and $(\alpha _1,+\infty )$ (resp. to $(0,\alpha _1)$,
    $(\alpha _1,\alpha _2)$ and $(\alpha _2,+\infty )$).
    For $d=6$ and $\tilde{c}=3$, we denote by analogy the
    quantities
    $u_1$, $u_2$, $u_3$ and $u_4$ with respect to the intervals
    $(0,\alpha _1)$, $(\alpha _1,\alpha _2)$, $(\alpha _2,\alpha _3)$
    and $(\alpha _3,+\infty )$. Example: the order of moduli $NNPNNN$
    corresponds to $[u_1,u_2]=[2,3]$ 
    while $NPNNPN$ corresponds to $[u_1,u_2,u_3]=[1,2,1]$ and 
    $NPPNNP$ corresponds to $[u_1,u_2,u_3,u_4]=[1,0,2,0]$. There are 6 couples
    $[u_1,u_2]$, $u_1+u_2=5$, 15 triples
    $[u_1,u_2,u_3]$, $u_1+u_2+u_3=4$, and 20 quadruples
    $[u_1,u_2,u_3,u_4]$, $u_1+u_2+u_3+u_4=3$.}
\end{nota}

For $d=6$, $\tilde{c}=1$, we list the orders of moduli
with which the sign patterns
$\Sigma _{m_1,m_2}$, $m_1+m_2=7$, are realizable (see~\cite{KoPuMaDe}):

\begin{equation}\label{eqc1}
  \begin{array}{ll}
 {\rm for}~1\leq m_1<m_2~,&0\leq u_2\leq 2m_1-2~;\\ \\ 
 {\rm for}~1\leq m_2<m_1~,&0\leq u_1\leq 2m_2-2~.\end{array}
\end{equation}
For $d=6$, $\tilde{c}=2$, realizability of couples has been studied
in~\cite{KoSoz19}.
There are two cases of canonical sign patterns. The corresponding couples are:

\begin{equation}\label{eqcan}
  (\Sigma _{1,5,1},~[0,4,0])~~~\, {\rm and}~~~\,
  (\Sigma _{m,1,q},~[q-1,0,m-1])~,~~~\, ~m+q=6~.
  \end{equation}
We give the remaining results in a table in which the first
column contains the sign
pattern, the second the realizable and the third the non-realizable triples
$[u_1,u_2,u_3]$:

$$\begin{array}{lcccc}
  {\rm SP}&&{\rm Y}&&{\rm N}\\ \\
  \Sigma _{2,4,1}&&[0,2,2],~[0,3,1],~[0,4,0]&&{\rm all~other~cases}\\ \\ 
  \Sigma _{3,3,1}&&[1,0,3],~[0,0,4],~[0,1,3],&&{\rm all~other~cases}\\
  &&[0,2,2],~[0,3,1],~[0,4,0]&&\\ \\
  \Sigma_{4,2,1}&&[1,0,3],~[0,0,4],~[0,1,3]&&{\rm all~other~cases}\\
  &&[0,2,2]&&\\ \\
  \Sigma_{2,3,2}&&{\rm all~possible~cases}&&{\rm no~cases}\\ \\
  \Sigma_{3,2,2}&&{\rm all~other~cases}&&[4,0,0],~[3,1,0]\\ &&&&[2,2,0],~[1,3,0]
  \end{array}$$

\subsection{The ratio between the numbers of realizable and all couples}

The number of realizable couples with $d=6$ and $\tilde{c}=3$ can be found
using Theorem~\ref{tmmain} and Remarks~\ref{remsBDG}. It equals

$$5\times 4+6\times 4+15\times 2+4\times 2+1\times 4+1\times 2+1\times 2=90~.$$
These products correspond to the orbits $A$, $C$, $E$, $F$, $B$, $D$ and $G$
respectively. The second factor corresponds to the number of sign patterns
in the given orbit.

At the same time the number of compatible orders of moduli with 3 letters $P$
and 3 letters $N$ equals 20. So the number of all couples with $d=6$ and
$\tilde{c}=3$ equals

$$(4+4+2+2+4+2+2)\times 20=400~.$$
For $\tilde{c}=0$ and $6$, the only couples $K:=(\Sigma_7,NNNNNN)$ and
$i_m(K)$ are realizable. For $\tilde{c}=1$, the numbers of realizable and
of all couples
are (see (\ref{eqc1}))

$$1+3+5+5+3+1=18~~~{\rm and}~~~\, 6\times 6=36~~~\, {\rm respectively~.}$$
The same numbers apply to the case $\tilde{c}=5$ as well (one has to use
the involution~$i_m$). The factor $6$ stands for the number of orders of
moduli with $\tilde{c}=1$ or~$5$.

For $\tilde{c}=2$ and $4$, we use the end of the previous subsection to find
these numbers. The last table shows that there are 4 orbits of sign patterns
of length 4 and 1 of length 2 each
with 15 compatible orders of moduli
(of which half correspond to the case $\tilde{c}=2$ and the other half to
$\tilde{c}=4$). This makes 270 couples. To these one has to add the
canonical sign patterns (see (\ref{eqcan})) which brings another
$6\times 15=90$ couples with
$\tilde{c}=2$ and $90$ with $\tilde{c}=4$. So there are 450 couples
of which 12 are realizable in the case of canonical sign patterns and
$3\times 4+6\times 4+4\times 4+15\times 2+11\times 4=126$ in the other cases.

Thus the ratio between the numbers of realizable and 
of all couples is

$$r(6)=(90+2+36+(12+126))/(400+2+72+450)=19/66~.$$
The numbers $r(d)$, $d\leq 5$, are computed in \cite{KoCMA23}. For $d\leq 6$,
the sequence of numbers $r(d)$ looks like this: $1$, $2/3$, $3/5$, $3/7$,
$47/126$, $19/66$. One could conjecture that this sequence (defined for
$d\in \mathbb{N}^*$) is decreasing. The sequence $r(d+1)/r(d)$, $d=1$,
$\ldots$, $5$, equals

$$2/3=0.66\ldots ,~~\, 9/10=0.9,~~\, 5/7=0.71\ldots ,~~\,
47/54=0.87\ldots ,~~\, 399/517=0.77\ldots ~.$$
It seems that when the ratio $r(d+1)/r(d)$ is defined for $d\in \mathbb{N}^*$,
this gives two adjacent sequences.

\subsection{The methods used in the proof of Theorem~\protect\ref{tmmain}
\protect\label{subsecmethods}}

We use four methods in the proof of Theorem~\ref{tmmain}. Three of them can be
qualified as analytic and the fourth as computational. In the next subsection
we explain how realizability of certain couples for degree $d+1$ hyperbolic
polynomials can be deduced from the realizability of couples for degree~$d$. 
The second method consists in proving that the inequalities between the moduli
of roots do not allow certain coefficients of a hyperbolic polynomial to
have certain signs. In Subsection~\ref{subsecneigh} we describe another 
method used to
prove that certain couples
are not realizable. The method is based on properties of the set $E_d$
of hyperbolic polynomials having a couple of non-zero opposite real roots.

Finally, in order to quickly obtain
examples of realizability, we use a program which generates uniformly
distributed random numbers. For
a given degree $d$, a given sign pattern and a given order of moduli,
the program generates $d$ real numbers
to create roots verifying the moduli order. Then the code calculates
the coefficients of the polynomial
and checks whether they match the sign pattern. If it is the case,
the code stops and returns the result, 
i.~e. the polynomial. If not, it continues and repeats the simulation
until it finds one or stops if the
number of simulations is reached. Finding concrete examples of realizability
when analytic methods fail turns out to be indispensable in the context of
a problem closely related to Question~\ref{quest2}, see~\cite{FoKoSh}.
The problem asks for real, but not necessarily hyperbolic polynomials, 
which triples (sign pattern, number of positive roots,
number of negative roots) compatible with Descartes' rule of signs, are
realizable.

\subsection{Concatenation of couples\protect\label{subsecconcat}}

Consider a hyperbolic degree $d$ polynomial $V$ with distinct moduli of roots
and non-vanishing coefficients. Denote by $\Omega$ the order of the moduli of
its roots, where $\Omega$ is a string of letters $P$ and/or $N$. Then for
$\varepsilon >0$ small enough, the first $d+1$ coefficients of the
degree $d+1$ hyperbolic polynomials
$W_-:=V(x)(x-\varepsilon )$ and $W_+:=V(x)(x+\varepsilon )$ are perturbations of
the respective coefficients of $V$. Hence they are of the same signs
as the latter coefficients. The three polynomials realize the couples

$$V:=(\sigma (V),\Omega )~,~~~\, W_-:(\sigma (W_-),P\Omega )~~~\,
{\rm and}~~~\, W_+:(\sigma (W_+),N\Omega )~.$$
Denote by $\alpha$ the last component of the sign pattern $\sigma (V)$, where
$\alpha =+$ or~$-$. Hence $\sigma (W_-)$ (resp. $\sigma (W_+)$)
is obtained from $\sigma (V)$ by adding to the right the component $-\alpha$
(resp. $\alpha$). We say that the couples $W_-$ and $W_+$ are obtained by
{\em concatenation} of the couple $V$ with the couples $((+,-),P)$ and
$((+,+),N)$ respectively. The method of concatenation is explained in a
broader context in \cite{FoKoSh} and within the framework of the 
problem mentioned at the end of Subsection~\ref{subsecmethods}.

\subsection{The set $E_d$ and neighbours of quadruples
  \protect\label{subsecneigh}}

\begin{nota}\label{notaPi}
  {\rm For a given sign pattern $\sigma$ of length $d+1$, we denote by
    $\Pi _d(\sigma )$
    the set of monic hyperbolic degree $d$ polynomials with distinct roots
    defining the sign
    pattern $\sigma$. For an order of moduli $\Omega$ compatible with $\sigma$,
    we denote by $\Pi _d(\sigma ,\Omega )\subset \Pi _d(\sigma )$ the
    set of monic hyperbolic degree $d$ polynomials defining the sign
    pattern $\sigma$ the order of moduli of their roots being~$\Omega$. We
    denote by $E_d(\sigma )$ the subset of $\Pi _d(\sigma )$ on which a
  positive and a negative root have equal moduli.}
\end{nota}

It is proved in \cite[Theorem~2]{KoAnn} that all sets of the form
$\Pi _d(\sigma )$ are open and contractible. It is shown in
\cite[Theorem~1.5]{GaKoTa} that at a generic point the set $E_d(\sigma )$
is locally a smooth hypersurface; at a point, where there are $s$ distinct
couples (positive root, negative root) of equal modulus, $E_d(\sigma )$ is
the transversal intersection of $s$ smooth hypersurfaces.

\begin{defi}\label{defineighbour}
  {\rm Two quadruples $[u_1,u_2,u_3,u_4]$ are {\em neighbours} if they are
    obtained from one another by transferring a unit one position to the left
    or right. E.~g. all neighbours of $[0,2,0,1]$ are $[1,1,0,1]$,
    $[0,1,1,1]$ and $[0,2,1,0]$.} 
  \end{defi}

\begin{prop}\label{propneighbours}
  Suppose that for given degree $d$ and sign pattern $\sigma$, two couples
  $(\sigma ,\Omega _1)$ and $(\sigma ,\Omega _2)$ are realizable, with
  $\Omega _1\neq \Omega _2$. Then there
  is a continuous path connecting two points $A_i\in \Pi _d(\sigma ,\Omega _i)$,
  $i=1$, $2$, passing through a point $A'\in \Pi _d(\sigma ,\Omega ')$, where
  $\Omega '$ is a neighbour of $\Omega _1$.
  \end{prop}

\begin{proof}
  Indeed, one can assume that the path $\gamma$ is smooth and avoids
  the non-generic
  points of $E_d(\sigma )$, i.~e. the points 
  at which there is more than one pair of opposite real non-zero roots, see
  \cite[Theorem~1.5]{GaKoTa}.
  On the other hand, as $\Omega _2\neq \Omega _1$, the path $\gamma$
  intersects $E_d(\sigma )$. The first time when this occurs the path
  passes from $\Pi _d(\sigma ,\Omega _1)$ into $\Pi _d(\sigma ,\Omega ')$,
  where $\Omega '$ is a neighbour of $\Omega_1$.  
\end{proof}

\begin{rem}\label{remneighbours}
  {\rm As the path from the proof of Proposition~\ref{propneighbours} avoids
    the non-generic points of $E_d$, when it 
    intersects the common border of the sets $\Pi _d(\sigma ,\Omega _1)$
    and $\Pi _d(\sigma ,\Omega ')$, this point corresponds to a polynomial
    having two opposite real roots (and this is the only equality between
    moduli of its roots). After a linear change of the variable $x$ this
    polynomial can be given the form $Q=(x^2-1)R$, where $R$ is a degree $d-2$
    monic hyperbolic polynomial.}
  \end{rem}

\section{Proof of Theorem~\protect\ref{tmmain}\protect\label{secprtmmain}}

\begin{proof}[Part~(1)]
  The couple $(\Sigma _{3,1,2,1},~PNPPNN)$ is realizable, because $PNPPNN$ is
  the canonical order of moduli (see \cite[Proposition~1]{KoSe}).
  We prove by examples the realizability of 
  the remaining 4 couples mentioned in
  part (1) of the theorem. Our examples involve polynomials having a positive
  and a negative root of equal moduli. After perturbing these roots so that
  they become of distinct moduli (the perturbation does not change the signs
  of the coefficients) one obtains polynomials realizing the given order
  of moduli with the sign pattern $\Sigma _{3,1,2,1}$. For the first
  polynomial of the list below the perturbed roots equal
  $-9$, $-1.01$, $-1-\varepsilon$, $0.39$, $0.4$ and $1-\varepsilon$,
  $0<\varepsilon \ll 0.1$. 

  $$\begin{array}{ll}PPPNNN &(x-0.39)(x-0.4)(x-1)(x+1)(x+1.01)(x+9)\\
&=x^6+8.23x^5+0.1902x^4-13.26928x^3\\ &+0.08276x^2+5.03928x-1.27296\\ \\
    PPNPNN &(x-0.2)(x-1)(x+1)(x-3.1)(x+5)(x+10)=\\
    &=x^6+11.7x^5+0.12x^4-167.4x^3+29.88x^2+155.7x-31\\ \\
    PPNNPN~~&(x-0.39)(x-0.4)(x+0.99)(x+1)(x-1)(x+9)\\
    &=x^6+9.2x^5+0.1739x^4-14.68046x^3\\ &+0.21606x^2+5.48046x-1.38996\\ \\
    NPPPNN~~&(x+1)(x-1)(x-2)(x-2.1)(x+5)(x+20)\\
    &=x^6+20.9x^5+0.7x^4-325.9x^3+418.3x^2+305x-420\end{array}$$
    Now we prove the non-realizability of the rest of the orders of moduli
    with the sign pattern $\Sigma _{3,1,2,1}$. Part of the results concern also
    the sign pattern $\Sigma _{3,2,1,1}$. 

    \begin{prop}\label{propu4=0}
      The $10$ orders of moduli with $u_4=0$ are not realizable with any of
      the sign patterns $\Sigma _{3,1,2,1}$ or $\Sigma _{3,2,1,1}$.
    \end{prop}

    \begin{proof}
      In the proof $\sigma$ denotes any of the sign patterns
      $\Sigma _{3,1,2,1}$ or $\Sigma _{3,2,1,1}$. The couples $(\sigma ,PPNNPN)$
      are realizable, see the examples at the beginning of the proofs of
      parts (1) and (4) of the theorem. For these couples one has $u_4=1$. 
      This implies that the set $\Pi _6(\sigma ,PPNNPN)$
      is open and non-empty, see Notation~\ref{notaPi}.
      Denote by $\Omega$ an order of moduli
      compatible with the sign pattern $\sigma$ and with $u_4=0$.
      If the set $\Pi _6(\sigma ,\Omega )$ is non-empty, then there
      exists a continuous path $\gamma \subset \Pi _6(\sigma )$
      leading from a point of $\Pi _6(\sigma ,\Omega )$ to a point
      of $\Pi _6(\sigma ,PPNNPN)$ (see Proposition~\ref{propneighbours}
      and its proof).
      Hence this path contains a
      polynomial $Q$ defining the sign pattern $\sigma$ and having
      a positive and a negative root of equal modulus. Moreover, its other
      roots are of smaller modulus.

      Set $Q:=x^6+\sum_{j=1}^5q_jx^j$. After a linear
change of the variable $x$, $Q$ takes the form $Q:=(x^2-1)R$,
where $R:=x^4+\sum_{j=0}^3a_jx^j$ is a degree $4$ hyperbolic polynomial
the moduli of whose roots are $<1$.

    \begin{lm}\label{lmQR}
      Suppose that the polynomials $Q:=x^6+\sum_{j=1}^5q_jx^j$ and
      $R:=x^4+\sum_{j=0}^3a_jx^j$ are hyperbolic, $Q:=(x^2-1)R$ and $Q$
      defines one of the
      sign patterns $\Sigma_{3,1,2,1}$ and $\Sigma_{3,2,1,1}$. Then $R$
      defines the sign pattern $\Sigma_{3,1,1}$.
\end{lm}

 \begin{proof}
   It is clear that

   $$Q=x^6+a_3x^5+(a_2-1)x^4+(a_1-a_3)x^3+(a_0-a_2)x^2-a_1x-a_0~.$$
   If $Q$ defines the sign pattern $\Sigma_{3,1,2,1}$ or $\Sigma_{3,2,1,1}$, then
   $a_0>0$, $a_1<0$ and $a_3>0$. If $a_2\leq 0$, then $a_2-1<0$ which
   contradicts each of the two sign patterns, so one must have $a_2>0$
   and $R$ defines the sign pattern $\Sigma_{3,1,1}$.
 \end{proof}

 The sign pattern $\Sigma _{3,1,1}$ is canonical. Hence the order of moduli
 defined by the roots of $R$ is $PPNN$. We denote these roots by $0<a<b$ and
 $-g<-f<0$, where $a<b<f<g$. Thus

 $$q_4=ab-af-ag-bf-bg+fg-1~.$$
 If $g\leq 1$, then $q_4=(fg-1)+a(b-f)-ag-bf-bg<0$ which is a contradiction. 
Thus all orders of moduli with $u_4=0$ are not realizable.

    \end{proof}

    We give the proof of non-realizability of the remaining $5$ couples.
    The couples $(\Sigma _{3,1,2,1},PNPNPN)$ and $(\Sigma _{3,2,1,1},PNPNPN)$
    (the order of moduli corresponds to the quadruple
    $[0,1,1,1]$) are not
    realizable, because the order of moduli $PNPNPN$ is rigid and hence
    realizable only with the sign pattern $\Sigma _{2,2,2,1}$, see
    Definition~\ref{defirigid} and the lines after it.

    Suppose that the order of moduli $[1,1,0,1]$ is realizable with the sign
    pattern $\Sigma _{3,1,2,1}$ or $\Sigma _{3,2,1,1}$, meaning
    that the following inequalities are satisfied: 

    $$\gamma_1<\alpha_1<\gamma_2<\alpha_2<\alpha_3<\gamma_3~.$$
    In this case, we have

\begin{equation}\label{eq1}\begin{array}{ccl}
q_4&=&(\alpha_1\alpha_2-\alpha_2\gamma_2)+(\alpha_1\alpha_3-\alpha_1\gamma_3)+
  (\alpha_2\alpha_3-\alpha_2\gamma_3) +(\gamma_1\gamma_2-\alpha_1\gamma_2)\\ &&+
(\gamma_1\gamma_3-\alpha_3\gamma_3)+(\gamma_2\gamma_3-\alpha_3\gamma_2)-
\alpha_1\gamma_1-\alpha_2\gamma_1-\alpha_3\gamma_1~.
\end{array}\end{equation}

    However, this is a sum of negative terms, which leads to a contradiction.

    Suppose  that the order of the moduli $[1,0,1,1]$ is realizable with
    the sign pattern $\Sigma _{3,1,2,1}$ or $\Sigma _{3,2,1,1}$, which
    means that the following inequalities are satisfied: 

$$ \gamma_1<\alpha_1<\alpha_2<\gamma_2<\alpha_3<\gamma_3.$$

 In this case, we obtain the following expression for $q_4$: 

 \begin{equation}\label{eq2}
   \begin{array}{ccl}

   q_4&=&(\alpha_1\alpha_2-\alpha_2\gamma_2)+(\alpha_1\alpha_3-\alpha_2\gamma_3)
   +(\alpha_2\alpha_3-\alpha_3\gamma_2) +(\gamma_1\gamma_2-\alpha_1\gamma_2)\\ 
   &&+(\gamma_1\gamma_3-\alpha_1\gamma_3)+(\gamma_2\gamma_3-\alpha_3\gamma_3)-
\alpha_1\gamma_1-\alpha_2\gamma_1-\alpha_3\gamma_1~.

\end{array}\end{equation}
 It is clear that $q_4$ is a sum of negative quantities,
 which is a contradiction.

 The orders of moduli $[2,0,0,1]$ and $[0,2,0,1]$ are not realizable with the
 sign pattern $\Sigma _{3,1,2,1}$, because neither of their neighbours is,
 see Proposition~\ref{propneighbours}. For $[2,0,0,1]$ these neighbours are
 $[1,1,0,1]$ and $[2,0,1,0]$. For
 $[0,2,0,1]$ they are $[1,1,0,1]$, $[0,1,1,1]$ and $[0,2,1,0]$. 
    
  \end{proof}

\begin{proof}[Part~(2)]
  The couple $(\Sigma_{2,1,2,2},NPNPPN)$ is realizable, because $NPNPPN$ is the
  canonical order. The other 5 couples mentioned in part (2) of the theorem
  are also realizable:

  $$\begin{array}{ll}
PNNPPN~~~&(x-4.52)(x+5.02)(x+5.32)(x-7.002)(x-8.003)(x+9.32)\\
&=x^6+0.135x^5-136.926694x^4+27.6529548x^3\\
&+5404.574382x^2-344.273285x-63044.12478\\ \\
NPPPNN&(x+2.5)(x-4.95)(x-6.47)(x-8.19)(x+8.57)(x+9.05)\\
&=x^6+0.51x^5-147.3884x^4+73.049286x^3\\
&+6188.991502x^2-7552.653247x-50858.41147\\ \\
NPPNPN&(x+1.49)(x-1.87)(x-5.77)(x+5.96)(x-7.58)(x+8.07)\\
&=x^6+0.3x^5-98.5114x^4+5.90954x^3\\
&+2380.426651x^2-720.0363792x-5861.282963\\ \\
NPPNNP&(x+1.34)(x-3.43)(x-5.34)(x+7.86)(x+9)(x-9.4)\\
&=x^6+0.03x^5-136.6074x^4+60.496052x^3\\
&+4547.732428x^2-6518.600281x-16320.4859\\ \\
NNPPPN&(x+2.5)(x+3.03)(x-4.28)(x-4.4)(x-5.6)(x+9.4)\\
&x^6+0.65x^5-86.2034x^4+122.15104x^3\\
&+1425.210824x^2-1478.768374x-7509.222336
    \end{array}$$
  We prove that the remaining 14 cases are not realizable.
  For the first 8 of them we suppose that they are realizable by a polynomial
  $Q:=x^6+\sum_{j=0}^5q_jx^j$. There are 4 cases in which one obtains that

  $$q_5:=(\gamma _1-\alpha _1)+(\gamma _2-\alpha _2)+(\gamma _3-\alpha _3)<0$$
  which contradicts the sign
  pattern. These are

$$\begin{array}{lcl}
\gamma _1<\gamma _2<\gamma _3<\alpha _1<\alpha_2<\alpha _3&:&[3,0,0,0] \\ 
\gamma _1<\gamma _2<\alpha _1<\gamma _3<\alpha _2<\alpha _3&:&[2,1,0,0] \\  
\gamma _1<\alpha _1<\gamma _2<\gamma _3<\alpha _2<\alpha _3&:&[1,2,0,0]\\ 
\gamma _1<\gamma _2<\alpha _1<\alpha _2<\gamma _3<\alpha _3&:&[2,0,1,0] 
  \end{array}$$
  There are 4 cases in which

  $$q_1:=\alpha _1\alpha _2\alpha _3\gamma_1\gamma_2\gamma_3(1/\alpha_1+
  1/\alpha_2+1/\alpha_3-1/\gamma_1-1/\gamma_2-1/\gamma_3)>0$$
  which also contradicts the sign pattern. The cases are:

  $$\begin{array}{lcl}
    \alpha_1<\alpha_2<\alpha_3<\gamma_1<\gamma_2<\gamma_3&:&[0,0,0,3] \\ 
    \alpha_1<\gamma_1<\alpha_2<\alpha_3<\gamma_2<\gamma_3&:&[0,1,0,2] \\ 
    \alpha_1<\alpha_2<\gamma_1<\alpha_3<\gamma_2<\gamma_3&:&[0,0,1,2] \\ 
    \alpha_1<\alpha_2<\gamma_1<\gamma_2<\alpha_3<\gamma_3&:&[0,0,2,1]
    \end{array}$$
The orders of moduli $[0,1,1,1]$ and $[1,1,1,0]$, i.~e. $PNPNPN$ and
$NPNPNP$ are rigid, see Definition~\ref{defirigid} and the lines after it,
hence non-realizable with the sign pattern
$\Sigma _{2,1,2,2}$.

In the four remaining cases of orders of moduli there exists at least one
realizable
neighbour and one can apply Proposition~\ref{propneighbours} and
Remark~\ref{remneighbours}. We list these cases to the left and their
neighbours to the right; non-realizable neighbours are marked by the index~0:

$$\begin{array}{rlcllll}
  1)&[0,0,3,0]&\hspace{1cm}&[0,1,2,0]~,&[0,0,2,1]_0&\\ \\
  2)&[0,1,2,0]&&[1,0,2,0]~,&[0,2,1,0]~,&[0,1,1,1]_0~,&[0,0,3,0]\\ \\
  3)&[0,2,1,0]&&[1,1,1,0]_0~,&[0,1,2,0]~,&[0,2,0,1]~,&[0,3,0,0]\\ \\
  4)&[0,3,0,0]&&[1,2,0,0]_0~,&[0,2,1,0]
  \end{array}$$
The four cases and their neighbours realizable with $\Sigma _{2,1,2,2}$
are all with $u_4=0$ or $u_4=1$.
Hence if one
applies Proposition~\ref{propneighbours} and Remark~\ref{remneighbours}, one
concludes that realizability of one of the cases 1)~--~4) implies the
existence of a polynomial $Q=(x^2-1)R$, where all roots of $R$ are of modulus
$<1$. Consider the orders of moduli $[0,0,3,0]$ and $[0,1,2,0]$. For the roots
$-g<-f<0<a<b$ of $R$ one should have

$$a<b<f<g<1~~~\, {\rm or}~~~\, a<f<b<g<1~~~\, {\rm respectively}~.$$
However, this would imply $q_1=abcd((1/a-1/f)+(1/b-1/g))>0$, which is a
contradiction. So the orders $[0,0,3,0]$ and $[0,1,2,0]$ are not realizable. 

Set $\sigma :=\Sigma _{2,1,2,2}$. Consider the set
$S:=S_1\cup S_2$, $S_1:=\Pi _6(\sigma ,[0,2,1,0])$,
$S_2:=\Pi _6(\sigma ,[0,3,0,0])$, see Notation~\ref{notaPi}. If at least one of
the set $S_1$ and $S_2$ is non-empty, then there exists a smooth path
$\gamma \subset \Pi _6(\sigma )$ connecting a point of $S$ with a point of
$\Pi _6(\sigma ,[0,2,0,1])$. One can choose the path avoiding the
non-generic points of the set $E_6$ and intersecting this set transversally.
Hence there exists a point of $\gamma$ belonging to the common border of
$S$ and $\Pi _6(\sigma ,[0,2,0,1])$. After a linear change of the variable $x$,
this point corresponds to a polynomial $Q=(x^2-1)R$, where for the roots of $R$
one has $a<f<g<b<1$ and 

$$\begin{array}{rrr}
  q_5:=-a-b+f+g>0~~~\, {\rm and}&q_1:=abfg(1/a+1/b-1/f-1/g)<0~,&
  {\rm i.~e.}\\ \\
a+b<f+g~~~\, {\rm and}&(a+b)/ab<(f+g)/fg~.
\end{array}$$
This, however, is impossible. Indeed, for fixed $f$, $g$ and $a+b$,
the quantity $(a+b)/ab$ is the minimal possible when $a$ and $b$ are closest
to one another. But for $b=g$, one should have $a<f$ and $1/a<1/f$ which
is contradictory; for $a=f$, this gives $b<g$ which is also a contradiction.
Hence $S=\emptyset$, i.~e. the orders of moduli $[0,2,1,0]$ and
$[0,3,0,0]$ are not realizable with~$\Sigma _{2,1,2,2}$.
\end{proof}

\begin{proof}[Part (3)]
  We prove that certain couples are realizable by concatenating couples
  corresponding to $d=5$ with ones corresponding to $d=1$, see
  Subsection~\ref{subsecconcat}. It is shown in
  \cite[Section~3]{KoCMA23} that for $d=5$, the sign pattern
  $\Sigma_{2,2,2}$ is realizable
  with all compatible orders $\Omega$
  ($\Omega$ is any string of 2 letters $P$ and 3 letters $N$). Applying the
  involution $i_m$ (see Definition~\ref{defiZ2Z2}) one sees that the sign
  pattern $i_m(\Sigma_{2,2,2})=\Sigma_{1,2,2,1}$ is realizable with any order
  $\Omega '$ which is a string of $3$ letters $P$ and $2$ letters $N$.
  
  Denote by $T$ (resp. $U$) a polynomial
  realizing the couple $(\Sigma_{2,2,2},\Omega )$ (resp.
  $(\Sigma_{1,2,2,1},\Omega ')$). Hence for $\varepsilon >0$
  small enough, the product $T(x)(x-\varepsilon )$ (resp.
  $U(x)(x+\varepsilon )$) realizes the order
  $P\Omega$ with the sign pattern $\Sigma_{2,2,2,1}$ (resp. the order $N\Omega '$
  with the sign pattern $\Sigma _{1,2,2,2}$), see Subsection~\ref{subsecconcat}.
  Applying the involution
  $i_r$ (see Definition~\ref{defiZ2Z2}) one understands that any order
  of the form $\Omega 'N$ is realizable with the sign pattern
  $\Sigma_{2,2,2,1}$.

  There are exactly $6$ orders which are not of the form $P\Omega$ or
  $\Omega 'N$. These are the $5$ orders mentioned in part (3) of the theorem
  and the order $NPPNNP$. The latter is realizable with the sign pattern
  $\Sigma _{2,2,2,1}$:

  $$\begin{array}{l}(x+4)(x-5)(x-6)(x+8.74)(x+9.41)(x-9.59)=\\ \\ 
    x^6+1.56x^5-165.7351x^4-145.848506x^3+\\ \\
    7833.610842x^2+24.186884x-94645.70472~.\end{array}$$

    The order of moduli $NPNPNP$ is rigid (see Definition~\ref{defirigid}),
    so realizable only with the sign pattern $\Sigma _{1,2,2,2}$.

    We prove the non-realizability of the remaining 5 orders. Consider a
    degree $6$ hyperbolic polynomial $Q:=x^6+\sum_{j=1}^{5}q_j x^j$ with
    distinct moduli of roots $\alpha _i$ and $\gamma _j$ and defining 
    the sign pattern $\Sigma_{2,2,2,1}$. Thus 

 $$Q=\prod_{i=1}^3(x-\alpha_i)(x+\gamma_i)~~~\, {\rm and}~~~\, 
q_1=\alpha_1\alpha_2\alpha_3\gamma_1\gamma_2\gamma_3S_1,$$
where 

\begin{equation*}
S_1:=(1/\alpha_{3}-1/\gamma_{3})+(1/\alpha_{2}-1/\gamma_{2})
+(1/\alpha_{1}-1/\gamma_{1})
\end{equation*}
If the orders of moduli $[1,2,0,0]$, $[2,0,1,0]$, $[2,1,0,0]$, and
$[3,0,0,0]$ are realizable, then the moduli of the roots satisfy respectively
the following inequalities:

$$\begin{array}{lll} \gamma_1<\alpha_1<\gamma_2<\gamma_3<\alpha_2<\alpha_3~,&

  &\gamma_1<\gamma_2<\alpha_1<\alpha_2<\gamma_3<\alpha_3~, \\
  \gamma_1<\gamma_2<\alpha_1<\gamma_3<\alpha_2<

\alpha_3&\rm{and}&\gamma_1<\gamma_2<\gamma_3<\alpha_1<\alpha_2<\alpha_3~.\\ 

\end{array}$$
Thus $S_1<0$. Therefore, we have $q_1<0$, which leads to a contradiction.

\end{proof}

\begin{proof}[Part (4)]
  The order of moduli $PPNPNN$ is the canonical order for the sign pattern
  $\Sigma _{3,2,1,1}$, so the corresponding couple is realizable, see
  \cite[Proposition~1]{KoSe}. The remaining 3 couples are realizable by
  perturbations of the
  following polynomials (see the beginning of the proof of part (1) of the
  theorem with the explanation about perturbations):

  $$\begin{array}{ll}
PPPNNN~~&(x-0.039)(x-0.4)(x-1)(x+1)(x+1.001)(x+4)\\
&=x^6+4.562x^5+0.824161x^4-6.2417404x^3\\
&-1.7616986x^2+1.6797404x-0.0624624\\ \\
PPNNPN&(x-0.09)(x-0.19)(x+0.8)(x+1)(x-1)(x+13)\\
&=x^6+13.52x^5+5.5531x^4-16.19602x^3\\ &-6.37526x^2+2.67602x-0.17784\\ \\
PNPPNN&(x-0.02)(x+1)(x-1)(x-3.1)(x+5)(x+20)\\
&=x^6+21.88x^5+21.062x^4-332.33x^3\\ &-15.862x^2+310.45x-6.2
    \end{array}$$
It was mentioned already (see Proposition~\ref{propu4=0}) that all orders of
moduli with $u_4=0$ are not realizable with the sign pattern
$\Sigma _{3,2,1,1}$. Also in the proof of part (1) we saw that the sign pattern
$\Sigma _{3,2,1,1}$ is not realizable with any of the orders of moduli
$[1,0,1,1]$ or $[1,1,0,1]$.

Non-realizability of the couple $(\Sigma _{3,2,1,1},PNPNPN)$ was proved
in the proof of part (1). The orders $[2,0,0,1]$ and $[0,2,0,1]$ are also
not realizable with $\Sigma _{3,2,1,1}$, because their respective neighbours
$[1,1,0,1]$, $[2,0,1,0]$ and $[1,1,0,1]$, $[0,1,1,1]$, $[0,2,1,0]$ are not,
see Proposition~\ref{propneighbours}.

It remains to prove the non-realizability of the couple
$(\Sigma _{3,2,1,1},NPPPNN)$. It corresponds to the quadruple $[1,0,0,2]$ and
has exactly two neighbours: $[0,1,0,2]$ and $[1,0,1,1]$ only the first of
which is realizable. Proposition~\ref{propneighbours} and
Remark~\ref{remneighbours} imply that if the couple is realizable, then there
exists a polynomial $Q:=(x^2-1)R$ such that for the roots
$a$, $b$, $-f$ and $-g$ of $R$ one has

$$0<1<a<b<f<g~.$$
Using the same notation as in the proof of part (1)
we observe that

$$q_2=(ab-1)fg+a(f-b)+ag+bf+bg>0$$
whereas one should have $q_2<0$. This contradiction implies
the non-realizability of the couple.
\end{proof}


\begin{thebibliography}{Dillo 83}



\bibitem{Ca} F.~Cajori, A history of the arithmetical methods 
of approximation
to the roots of numerical equations of one unknown quantit



  \bibitem{Cu} D.~R.~Curtiss, Recent extensions of Descartes' rule of signs,
    Annals of Mathematics. 19 (4), 251-278 (1918).

  \bibitem{DG} J.-P.~de Gua de Malves, D\'emonstrations de la R\`egle
  de Descartes,
Pour conno\^{\i}tre le nombre des Racines positives \& n\'egatives dans
les \'Equations qui n’ont point de Racines imaginaires, Memoires de
Math\'ematique et de Physique tir\'es des registres de l’Acad\'emie Royale
des Sciences 72-96 (1741).

\bibitem{Des}
  The Geometry of Ren\'e Descartes with a facsimile of the first edition,
  translated by D. E. Smith and M.L. Latham, New York, Dover Publications,
  1954. 

\bibitem{FoKoSh} J.~Forsg\aa rd, V.~P.~Kostov and B.~Shapiro:
Could Ren\'e Descartes have known this?  Exp. Math. 24 (4)
(2015),
438-448. Zbl 1326.26027, MR3383475


\bibitem{Fo} J.~Fourier, Sur l'usage du th\'eor\`eme de Descartes
 dans la recherche des limites des racines. Bulletin des sciences
 par la Soci\'et\'e philomatique de Paris (1820) 156--165, 181--187;
 {\oe}uvres 2,  291--309, Gauthier-Villars, 1890.

\bibitem{GaKoTa} Y.~Gati, V.~P.~Kostov and M.~C.~Tarchi, Sign patterns
  and rigid moduli orders, The Graduate Journal of Mathematics, Volume 6,
  Issue 1 (2021), 60-72
 
\bibitem{Ga} C.~F.~Gauss, Beweis eines algebraischen Lehrsatzes. 
J. Reine Angew. Math. 3, 1-4 (1828); Werke 3, 67--70, 
G\"ottingen, 1866.




\bibitem{J} J.~L.~W.~Jensen, Recherches sur la th\'eorie des \'equations,
  Acta Mathematica 36, 181-195 (1913).






\bibitem{KoPuMaDe} V.~P.~Kostov, Descartes' rule of signs and moduli of roots,
 Publicationes Mathematicae Debrecen 96/1-2 (2020) 161-184,

\bibitem{KoSe} V.~P.~Kostov, Hyperbolic polynomials and canonical sign
  patterns, Serdica Math. J.  46 (2020) 135-150, arXiv:2006.14458.

\bibitem{KorigMO} V.~P.~Kostov, Hyperbolic polynomials and
  rigid moduli orders, 
  Publicationes Mathematicae Debrecen 100 (1-2) 119-128 (2022). 

  

\bibitem{KoAnn} V.~P.~Kostov, Univariate polynomials and the contractibility
  of certain sets, Annual of Sofia University ``St. Kliment Ohridski'',
  Faculty of Mathematics and Informatics 107 (2020), 11-35.

\bibitem{KoReMa22}  V.~P.~Kostov, Which Sign Patterns are Canonical?
  Results Math 77 (2022) No 6,
  paper 235.
  https://doi.org/10.1007/s00025-022-01769-3

\bibitem{KoCMA23} V.~P.~Kostov, Beyond Descartes' rule of signs,
  Constructive Mathematical Analysis, vol. 6 issue 2 (2023) 128-141.




\bibitem{KoSoz19} V.~P.~Kostov, Moduli of roots of hyperbolic polynomials
  and Descartes' rule of signs,
  Constructive Theory of Functions, Sozopol 2019
  (B. Draganov, K. Ivanov, G. Nikolov and R. Uluchev, Eds.), pp. 131-146
  Prof. Marin Drinov Academic Publishing House, Sofia, 2020.
  




\bibitem{La} E.~Laguerre, Sur la th\'eorie des \'equations num\'eriques,
  Journal de
  Math\'ematiques pures et appliqu\'ees, s. 3, t. 9, 99-146 (1883);
  {\oe}uvres 1,
  Paris, 1898, Chelsea, New-York, 1972, pp. 3–47.


\bibitem{Mes} B.~E.~Meserve, Fundamental Concepts of Algebra, New York,
  Dover Publications, 1982.

\end{thebibliography}
\end{document}